\documentclass[12pt,a4paper,reqno]{amsart}
\usepackage[english]{babel}
\usepackage[applemac]{inputenc}
\usepackage[T1]{fontenc}
\usepackage{mathtools}
\DeclarePairedDelimiter{\ceil}{\lceil}{\rceil}
\usepackage{palatino}
\usepackage{amsmath}
\usepackage{amssymb}
\usepackage{amsthm}
\usepackage{amsfonts}
\usepackage{graphicx}
\usepackage[colorlinks = true, citecolor = black]{hyperref}
\author{Tuomas Orponen}
\title{On the tube-occupancy of sets in $\R^{d}$}
\address{School of Mathematics, University of Edinburgh}
\thanks{T.O. gratefully acknowledges the financial support of the Finnish foundation Jenny and Antti Wihurin Rahasto.}
\email{tuomas.orponen@helsinki.fi}

\newcommand{\R}{\mathbb{R}}
\newcommand{\N}{\mathbb{N}}

\newcommand{\tn}{\mathbb{P}}

\newcommand{\calH}{\mathcal{H}}

\newcommand{\calS}{\mathcal{S}}
\newcommand{\calQ}{\mathcal{Q}}

\newcommand{\calU}{\mathcal{U}}

\newcommand{\calP}{\mathcal{P}}

\newcommand{\E}{\mathbb{E}}

\newcommand{\1}{\mathbf{1}}

\newcommand{\card}{\operatorname{card}}

\newcommand{\Var}{\operatorname{\mathbf{Var}}}

\numberwithin{equation}{section}

\theoremstyle{plain}
\newtheorem{thm}[equation]{Theorem}

\newtheorem{lemma}[equation]{Lemma}

\theoremstyle{definition}

\theoremstyle{remark}

\begin{document}

\begin{abstract} Call a pair $(s,t) \in [0,d] \times [0,d]$ \emph{admissible}, if there exists a compact set $K \subset \R^{d}$ and a constant $C > 0$ such that $0 < \calH^{s}(K) < \infty$, and
\begin{displaymath} \calH_{s}(K \cap T) \leq Cw(T)^{t} \end{displaymath}
for all tubes $T \subset \R^{d}$ of width $w(T)$. The purpose of this paper is to show that all pairs $(s,s)$ with $s < d - 1$ are admissible. Combined with previous results, this settles a question of A. Carbery.
\end{abstract}

\maketitle

\section{Introduction}

Let $d \geq 2$ be an integer. In \cite{Ca}, A. Carbery asks to determine which pairs $(s,t) \in [0,d] \times [0,d]$ are \emph{admissible} in the sense that there exists a (compact) set $K \subset \R^{d}$ of positive and finite $s$-dimensional Hausdorff measure with the property that
\begin{equation}\label{tubes} \calH_{s}(K \cap T) \leq Cw(T)^{t} \end{equation}
for all tubes $T$ and for some constant $C > 0$. Here and below, a \emph{tube $T \subset \R^{d}$ of width $w(T) > 0$} refers to the $w(T)/2$-neighbourhood of a line in $\R^{d}$. The known results can be summarised as follows:
\begin{itemize}
\item The condition \eqref{tubes} is closely related to the projections of the measure $\calH^{s}|_{K}$ into $(d - 1)$-dimensional subspaces. In particular, \eqref{tubes} implies that the projections of $K$ onto any such subspace are at least $t$-dimensional. This means that only pairs $(s,t)$ with $t \leq \min\{d - 1, s\}$ have a chance of being admissible.
\item Conversely, it was shown in \cite{Ca}, it was shown that that all pairs $(s,t)$ with $t < \min\{d - 1,s\}$ are admissible.

\item It follows from the Besicovitch projection theorem that the pair $(d - 1, d - 1)$ is \textbf{not} admissible; for details, see \cite{CSV}.
\item On the other hand, all pairs $(s,d - 1)$ are admissible for $s > d - 1$. This is due to P. Shmerkin and V. Suomala \cite{SS}. 
\end{itemize} 
So, the only remaining question concerns the pairs $(s,s)$ with $s < d - 1$. We answer this question:
\begin{thm}\label{main} All pairs $(s,s)$ are admissible for $s < d - 1$.
\end{thm} 

We will present an informal overview of the proof right away, while the technical details are given in the third and fourth sections. To prove the theorem, we need to build a set $K$ with $0 < \calH^{s}(K) < \infty$, satisfying \eqref{tubes} with $t = s$. To this end, we use a standard Cantor type construction: for various "generations" $n \in \N$, we seek for collections of closed and disjoint cubes $\calQ_{n}$ such that the union of the cubes in $\calQ_{n + 1}$ is contained in the union of the cubes in $\calQ_{n}$. For fixed $n \in \N$, the cubes in $\calQ_{n}$ will have a common side-length $\ell_{n} > 0$. As an induction hypothesis, we may assume that every tube $T \subset \R^{d}$ of width $w(T) = \ell_{n}$ intersects at most $k$ cubes in $\calQ_{n}$ for some large $k \in \N$. Then, to proceed, we need to find the family of cubes $\calQ_{n + 1}$ containing $\sim \ell_{n + 1}^{-s}$ members, so that every tube of width $\ell_{n + 1}$ meets no more than $k$ of them. Moreover, we have to do this in such a manner that all tubes $T$ of "intermediary" widths $\ell_{n + 1} < w(T) \leq \ell_{n}$ also behave well -- that is, do not intersect too many cubes in $\calQ_{n + 1}$.

The first idea would be to throw the centres of the cubes in $\calQ_{n + 1}$ uniformly at random inside the union of the cubes in $\calQ_{n}$. This cannot work directly, since random sets contain a certain amount of clustering with high probability (WHP), and \eqref{tubes} means practically zero tolerance towards such behaviour. Fortunately, this "certain amount" turns out to be small and easily quantifiable. In particular, WHP, there are only a few clusters of the type where some tube of width of with $w(T) = \ell_{n + 1}$ intersects more than $k$ cubes in $\calQ_{n + 1}$. So, we may simply take the clusters one by one and remove points from them, until they are clusters no longer. And, WHP, as it turns out, we can do this so that at least half of the originally selected points remain. A version of this procedure appeared already in \cite{Ca}, and, much earlier, a similar idea was used in connection with the Heilbronn triangle problem by Koml\'os, Pintz and Szemer\'edi \cite{KPS}.

Next, we face the tubes of width $\ell_{n + 1} < w(T) \leq \ell_{n}$. After defining the appropriate notion of "clustering" for tubes of intermediary width, it turns out that there are, again, only a few clusters WHP. So, fixing $w(T) \in (\ell_{n + 1},\ell_{n}]$ we may remove some further points to get rid of these clusters in the same manner as above. In fact, the expected number of clusters is so negligible that we can get simultaneously rid of \textbf{all} clusters corresponding to tubes of width $w(T) \in (\ell_{n + 1},\eta_{n + 1}]$, where $\eta_{n + 1} \in (\ell_{n + 1},\ell_{n})$ is a certain constant.

For tubes $T$ of width $w(T) \in (\eta_{n + 1},\ell_{n}]$ the strategy does not seem to work directly. Instead, the number $\eta_{n + 1}$ is selected so that every cube of side-length $\eta_{n + 1}$ is expected to contain $\sim 1$ of our randomly selected points. In this situation, we can redefine a "cluster" to mean an $A$-element set contained in a single cube of side-length $\eta_{n + 1}$, where $A \in \N$ is a large constant. Choosing $A$ large enough, it turns out that there are only a few clusters of this type, and they can be disposed of in a familiar manner. After this is done, we end up with a set, which is well-behaved with respect to thin tubes -- those of width $\ell_{n + 1} < w(T) \leq \eta_{n + 1}$ -- and is also very uniformly distributed on scales larger than $\eta_{n + 1}$. Finally, it suffices to check that \textbf{any} set, which has the latter uniform distribution property, is well-behaved with respect to thick tubes, namely those of width $\eta_{n + 1} < w(T) \leq \ell_{n}$. This concludes the inductive construction of $\calQ_{n + 1}$. 

Before we start with the technicalities, some quick words on notation and presentation: the letters $Q,R,S$ will be used to denote various cubes (with sides perpendicular to coordinate axes), while $T$ always stands for a tube -- a $w(T)/2$-neighbourhood of a line in $\R^{d}$, where $w(T) > 0$ is the \emph{width} of $T$.

Given $A,B > 0$, the notation $A \lesssim B$ means that $A \leq CB$ for some constant $C$ depending only on the dimension $d$ of the ambient space (so $d$ is regarded as an absolute constant in our notation). The two-sided inequality $A \lesssim B \lesssim A$ is abbreviated to $A \sim B$. Both cardinality and Lebesgue measure of planar sets will be denoted by $| \cdot |$: the notation will refer to cardinality for finite sets, and to Lebesgue measure otherwise. 


\section{Acknowledgements}

I am grateful to Henna Koivusalo for fruitful discussions during the early stages of the paper's preparation. I am also thankful to Tony Carbery for suggesting the problem and giving comments on the manuscript. Finally, I thank the referee for a careful review and many clarifying comments.

\section{Main lemma}\label{technicalSection}

As we outlined above, the set $K$ needed for Theorem \ref{main} will be constructed by defining recursively the families of cubes $\calQ_{n}$. The initial family will be $\calQ_{0} := \{[0,1]^{d}\}$, so all further action will take place inside the unit cube. The cubes have to satisfy certain properties, which are listed in the following main lemma:

\begin{lemma}\label{mainLemma} Fix $s < d - 1$ and $\delta > 0$ such that $\delta^{-s} \in \N$. Then, the following holds for large enough $k = k_{s}, m = m_{\delta,s} \in \N$ (so $k$ depends only on $s$, while $m$ may also depend on $\delta$). Assume that $U \subset [0,1]^{d}$ is the union of a family $\calU$ of $\delta^{-s}$ disjoint cubes, each of side-length $0 < \delta \leq 1$, such that every tube of width $2\delta$ meets no more than $k$ cubes. Then, there exists a collection of disjoint closed cubes $\calQ$ with the following properties.
\begin{itemize}
\item[(a)] All the cubes in $\calQ$ have equal side-length $\epsilon \sim 1/m$.
\item[(b)] The union of the cubes in $\calQ$ is contained in the union of the cubes in $\calU$. Moreover, every cube in $\calU$ contains exactly $(\delta/\epsilon)^{s} \in \N$ cubes in $\calQ$. In particular, there are $\epsilon^{-s} \in \N$ cubes in $\calQ$ altogether.
\item[(c)] An arbitrary cube of side-length $\delta^{(d - s)/d}m^{-s/d}$ intersects at most one cube in $\calQ$.
\item[(d)] Every tube $T$ of width $w(T) = 2\epsilon$ satisfies
\begin{displaymath} \card\{Q \in \calQ : T \cap Q \neq \emptyset\} \leq k, \end{displaymath}
and every tube $T$ of width $\epsilon \leq w(T) \leq \delta$ satisfies
\begin{displaymath} \card\{Q \in \calQ : T \cap Q \neq \emptyset\} \lesssim k \cdot \left(\frac{w(T)}{\epsilon}\right)^{s}. \end{displaymath}
\end{itemize} 
\end{lemma}

For the remainder of Section \ref{technicalSection}, the parameters $\delta,\epsilon,k,m,s,\calQ,U$ and $\calU$ refer to those introduced in the statement of Lemma \ref{mainLemma}. The constant $\eta_{n + 1}$ discussed in the introduction is not explicitly mentioned, but it coincides with $\delta^{(d - s)/d}m^{-s/d}$. The proof idea is to pick $m^{s}$ points independently and uniformly at random inside $U$, for some large $m \in \N$, and call this random set $P_{0}$ (if $m^{s} \notin \N$, use $\ceil{m^{s}}$ instead). The set $P_{0}$ will play the role of the centres of the cubes in $\calQ$. In the interest of avoiding extra constants, we choose to ignore the issue of the points in $P_{0}$ being too close to the boundary of $U$ -- resulting in the risk that some cubes in $\calQ$ may not be entirely contained in the union of the cubes of $\calU$. The risk could be neutralised by first replacing the cubes in $\calU$ by slightly smaller ones and then choosing the points $P_{0}$ inside the new cubes.

\subsection{Preliminary considerations towards (c)} Divide each of the $\delta^{-s}$ cubes $R \in \mathcal{U}$ into a grid of $(\delta m)^{s}$ equally sized subcubes using equally spaced hyperplanes perpendicular to the coordinate axes (if $(\delta m)^{s} \notin \N$, use $\ceil{(\delta m)^{s}}$ instead). Denote the collections of cubes so obtained by $\calS_{R}$, $R \in \calU$. 
\begin{lemma}\label{mainClaim} Let $A \in \N$, and for each $R \in \calU$, consider the random variable
\begin{displaymath} X_{R} := \frac{1}{(\delta m)^{s}} \sum_{S \in \calS_{R}} \binom{|P_{0} \cap S|}{A}, \end{displaymath}
where $P_{0}$ is the random $m^{s}$-element subset of $U$, and we make the usual convention that the binomial coefficient is zero, when $A > |P_{0} \cap S|$. Then, 
\begin{displaymath} \tn\left\{\max_{R \in \calU} X_{R} \geq \frac{1}{10} \right\} < \frac{1}{10}, \end{displaymath}
if $A$ is large enough (but absolute), and $m \in \N$ is large enough (depending on $\delta$).
\end{lemma}

\begin{proof}[Proof of Lemma \ref{mainClaim}] The plan is to use the union bound
\begin{displaymath} \tn\left\{\max_{R \in \calU} X_{R} \geq \frac{1}{10}\right\} \leq \sum_{R \in \calU} \tn\{X_{R} \geq 1/10\} \end{displaymath}
in the end, so we need to show that $\tn\{X_{R} \geq 1/10\}$ can be pushed smaller than $1/ (10 \card \calU) = \delta^{s}/10$ by taking $A$ and $m$ large enough. First, we observe that $X_{R}$ is the average over the random variables
\begin{displaymath} Y_{S} := \binom{|P_{0} \cap S|}{A}, \end{displaymath}
so we wish to see that (a) the expectation of the variables $Y_{S}$ can be made small by taking $A$ large, and (b) the average $X_{R}$ is strongly concentrated around the expectation. The only problem in (b) is that the random variables $Y_{S}$ are not independent: in fact, if $Y_{S}$ is large for some particular $S \subset R$, then there are unexpectedly many elements of $P_{0}$ packed inside $S$, and this makes it less likely that $Y_{S'}$ is large for $S' \neq S$. In fact, this observation is the key to the proof.

Fixing $R \in \calU$, we make the preliminary estimate
\begin{align} \tn & \left\{X_{R} \geq \frac{1}{10}\right\} \leq \tn \left\{|P_{0} \cap R| > 2(\delta m)^{s} \right\}\notag\\
&\label{form12} \quad\quad + \tn \left\{X_{R} \geq \frac{1}{10} \text{ and } |P_{0} \cap R| \leq 2(\delta m)^{s} \right\}. \end{align}
Since the random variable $|P_{0} \cap R|$ is distributed $\text{Bin}(m^{s},\delta^{s})$, the probability of the first summand tends to zero as $m \to \infty$. So, we are left to deal with the second summand. This probability can be further expressed as the following sum of conditional probabilities:
\begin{equation}\label{fixk} \eqref{form12} = \sum_{k = 0}^{2(\delta m)^{s}} \tn \left\{X_{R} \geq \frac{1}{10} \;\Big|\; |P_{0} \cap R| = k \right\}\tn\left\{|P_{0} \cap R| = k\right\}. \end{equation}
So, we aspire to estimate 
\begin{displaymath} \tn_{k} \left\{X_{R} \geq \frac{1}{10} \right\}, \qquad k \leq 2(\delta m)^{s}, \end{displaymath}
where $\tn_{k}\{\cdots\} = \tn\{\cdots \mid |P_{0} \cap R| = k\}$. Now, even though $X_{R}$ was initially defined on the probability space associated with the entire set $P_{0}$, the conditioning on "$|P_{0} \cap R| = k$" allows us to restrict attention inside the cube $R$. More precisely, we can consider the new probability space associated with "throwing $k$ points uniformly and independently at random inside $R$" and observe that the distribution of $X_{R}$ with respect to this new (unconditional) probability equals the $\tn_{k}$-distribution of $X_{R}$. The random $k$-element subset of $R$ will be denoted by $P_{0}^{k}$, and we will keep the notation $\tn_{k}$ for the probability measure associated with $P_{0}^{k}$.

We start to bound the probability $\tn_{k}[X_{R} \geq 1/10]$. If $\calS \subset \calS_{R}$ is a collection of cubes, and $(r_{S})_{S \in \calS}$ is a collection of natural numbers, denote by $E((r_{S})_{S \in \calS})$ the event
\begin{displaymath} E((r_{S})_{S \in \calS}) = \left\{|P_{0}^{k} \cap S| = r_{S} : S \in \calS\right\}. \end{displaymath}
Then the event $\{X_{R} \geq 1/10\}$ is contained in the finite disjoint union of events of the form $E((r_{S})_{S \in \calS})$ such that $r_{S} \geq A$ for all $S \in \calS$, and
\begin{equation}\label{sumrS} \sum_{S \in \calS} \binom{r_{S}}{A} \geq \frac{(\delta m)^{s}}{10}. \end{equation}
Fix one such event $E((r_{S})_{S \in \calS})$, and enumerate the cubes in $\calS$ by writing $\calS = \{S_{1},\ldots,S_{N}\}$. Then
\begin{align}\label{tnE} \tn_{k} & [E((r_{S})_{S \in \calS})] = \tn_{k}[\{|P_{0}^{k} \cap S_{1}| = r_{S_{1}}\} \cap \ldots \cap \{|P_{0}^{k} \cap S_{N}| = r_{S_{N}}\}]\\
& = \prod_{n = 1}^{N} \tn_{k}[\{|P_{0}^{k} \cap S_{n}| = r_{S_{n}}\} \mid \{|P_{0}^{k} \cap S_{n + 1}| = r_{S_{n + 1}}\} \cap \ldots \cap \{|P_{0}^{k} \cap S_{N}| = r_{S_{N}}\}], \notag \end{align}
simply by iterating the definition of conditional probability. Let us fix $n \in \{1,\ldots,N\}$ and study the $n^{th}$ factor of the product. The distribution of the random variable $|P_{0}^{k} \cap S_{n}|$ with respect to the conditional probability measure
\begin{displaymath} \tn_{k}[ \: \cdots \mid \{|P_{0}^{k} \cap S_{n + 1}| = r_{S_{n + 1}}\} \cap \ldots \cap \{|P_{0}^{k} \cap S_{N}| = r_{S_{N}}\}] \end{displaymath}
is a binomial one, more precisely
\begin{displaymath} \text{Bin}\left(k - \sum_{q = n + 1}^{N} r_{S_{q}}, \frac{|S_{n}|}{|R| - (N - n)|S_{n}|} \right). \end{displaymath}
Indeed, given the information that $r_{S_{n + 1}} + \ldots + r_{S_{N}}$ points are contained in the cubes $S_{n + 1},\ldots,S_{N}$, we are left with $k - r_{S_{n + 1}} - \ldots - r_{S_{N}}$ points to choose randomly inside the set 
\begin{displaymath} R \setminus \bigcup_{q = n + 1}^{N} S_{q}, \end{displaymath}
and the probability that any of these points should land in $S_{n}$ is exactly the "success probability" of the binomial distribution above. Now, we wish to bound the probability that such a random variable takes the value $r_{S_{n}}$. Applying the bound $\binom{n}{k} \leq (e\cdot n/k)^{k}$, this probability can be estimated as
\begin{align} \binom{k - \sum_{q = n + 1}^{N} r_{S_{q}}}{r_{S_{n}}} & \left(\frac{|S_{n}|}{|R| - (N - n)|S_{n}|}\right)^{r_{S_{n}}}\left(1 - \frac{|S_{n}|}{|R| - (N - n)|S_{n}|}\right)^{k - \sum_{q = n}^{N} r_{S_{q}}} \notag\\
& \leq \frac{e^{r_{S_{n}}}\left(k - \sum_{q = n + 1}^{N} r_{S_{q}}\right)^{r_{S_{n}}}}{r_{S_{n}}^{r_{S_{n}}}}\left(\frac{|S_{n}|}{|R| - (N - n)|S_{n}|}\right)^{r_{S_{n}}} \notag \\
& = \frac{(e|S_{n}|)^{r_{S_{n}}}}{r_{S_{n}}^{r_{S_{n}}}} \cdot \left(\frac{k - \sum_{q = n + 1}^{N} r_{S_{q}}}{|R| - (N - n)|S_{n}|}\right)^{r_{S_{n}}} \notag \\
&\label{form8} \leq \frac{(e|S_{n}|)^{r_{S_{n}}}}{r_{S_{n}}^{r_{S_{n}}}} \cdot \left(\frac{k - (N - n)A}{|R| - (N - n)|S_{n}|}\right)^{r_{S_{n}}}, \end{align}
using the assumption that $r_{S} \geq A$ for all $S \in \calS$. Next, a simple manipulation shows that
\begin{displaymath} \frac{k - (N - n)A}{|R| - (N - n)|S_{n}|} \leq \frac{k}{|R|}, \end{displaymath}
as soon as $A \geq k|S_{n}|/|R|$, which is true as soon as $A \geq 2$, since $k \leq 2(\delta m)^{s} = 2|R|/|S_{n}|$. Plugging this estimate into \eqref{form8} leads to
\begin{equation}\label{form9} \eqref{form8} \leq \frac{k^{r_{S_{n}}}}{r_{S_{n}}^{r_{S_{n}}}} \cdot \left(\frac{e|S_{n}|}{|R|}\right)^{r_{S_{n}}} \leq \binom{k}{r_{S_{n}}} \cdot \left(\frac{e|S_{n}|}{|R|}\right)^{r_{S_{n}}}. \end{equation}
This form is almost what we were looking for. The final observation is that
\begin{displaymath} \left(1 - \frac{2e|S_{n}|}{|R|} \right)^{k - r_{S_{n}}} \geq \left(1 - \frac{2e|S_{n}|}{|R|} \right)^{2(\delta m)^{s}} \geq \tau \end{displaymath}
for some absolute constant $\tau > 0$, since $|S_{n}|/|R| = (\delta m)^{-s}$. Thus, by taking $A$ so large that $2^{r_{S_{n}}} \geq 2^{A} \geq 1/\tau$, the estimate in \eqref{form9} can be taken one step further as follows:
\begin{displaymath} \eqref{form9} \leq \binom{k}{r_{S_{n}}} \cdot \left(\frac{2e|S_{n}|}{|R|} \right)^{r_{S_{n}}} \cdot \left(1 - \frac{2e|S_{n}|}{|R|}\right)^{k - r_{S_{n}}}. \end{displaymath}
Now, we look back to \eqref{tnE}. The probability of the event $E((r_{S})_{S \in \calS})$ was factorised into a product of $N$ numbers, and we have just obtained an estimate for each of these. Consequently,
\begin{equation}\label{form10} \tn_{k} [E((r_{S})_{S \in \calS})] \leq \prod_{n = 1}^{N} \binom{k}{r_{S_{n}}} \cdot \left(\frac{2e|S_{n}|}{|R|} \right)^{r_{S_{n}}} \cdot \left(1 - \frac{2e|S_{n}|}{|R|}\right)^{k - r_{S_{n}}}. \end{equation}
The estimate \eqref{form10} is useful once interpreted correctly. To do this, we associate to each cube $S \in \calS_{R}$ an independent copy of an abstract $\sim \text{Bin}(k,2e|S|/|R|)$ distributed random variable, which we denote by $Z_{S}$.\footnote{So, the variable $Z_{S}$ does not actually have anything to do with the cube $S \in \calS_{R}$, but such an indexing is handy nevertheless.} Given any collection of cubes $\calS = \{S_{1},\ldots,S_{N}\} \subset \calS_{R}$, and any natural numbers $r_{S_{n}} \in \{1,\ldots,k\}$ for $1 \leq n \leq N$, the probability of the event
\begin{displaymath} E_{Z}((r_{S})_{S \in \calS}) = \{Z_{S_{n}} =  r_{S_{n}} \text{ for all } 1 \leq n \leq N\} \end{displaymath}
equals the right hand side of \eqref{form10}. Moreover, assuming that the numbers $r_{S_{n}}$, $1 \leq n \leq N$ satisfy \eqref{sumrS}, the event $E_{Z}((r_{S})_{S \in \calS})$ is contained in
\begin{displaymath} \left\{\frac{1}{(\delta m)^{s}} \sum_{S \in \calS_{R}} \binom{Z_{S}}{A} \geq \frac{1}{10} \right\}. \end{displaymath}
The conclusion is that 
\begin{equation}\label{form11} \tn_{k}\left\{X_{R} \geq \frac{1}{10}\right\} \leq \tilde{\tn}\left\{\frac{1}{(\delta m)^{s}} \sum_{S \in \calS_{R}} \binom{Z_{S}}{A} \geq \frac{1}{10}\right\}, \end{equation}
where we used $\tilde{\tn}$ to denote the abstract probability measure associated with the random variables $Z_{S}$. Let us sum up the argument that lead to this conclusion: the event on the left hand side of \eqref{form11} can be expressed as the union of events of the form $E((r_{S})_{S \in \calS})$, where the numbers $r_{S}$ satisfy $r_{S} \geq A$ and \eqref{sumrS}. Then, the $\tn_{k}$-probability of any such event is bounded by the $\tilde{\tn}$-probability of the corresponding event $E_{Z}((r_{S})_{S \in \calS})$ by \eqref{form10}. Finally, the (disjoint) union of these events is contained in the event on the right hand side of \eqref{form11}.

Next, we plan to apply Markov's inequality to show that for $A$ large enough but absolute, and for $m$ large enough depending on $\delta$, the right hand side probability in \eqref{form11} is less than $\delta^{s}/20$. So, we need to estimate the expectation and variance of the random variables $\binom{Z_{S}}{A}$. The expectation can be calculated rather explicitly by viewing $\binom{Z_{S}}{A}$ as the sum of certain other random variables (this may seem complicated, but thinking along these lines will also be useful later on in the paper). Fix a cube $U \subset R$ of volume $|U| = 2e|S|$. Then, the $\tn_{k}$-distribution of $|P_{0}^{k} \cap U|$ is $\sim \text{Bin}(k,|U|/|R|)$, which is the same as the $\tilde{\tn}$-distribution of $Z_{S}$. Consequently, the $\tn_{k}$-distribution of the random variable
\begin{displaymath} \sum_{P \in \calP^{A}} \1_{\{P \subset U\}} = \binom{|P_{0}^{k} \cap U|}{A} \end{displaymath}
is the same as the $\tilde{\tn}$-distribution of the random variable $\binom{Z_{S}}{A}$, where $\calP^{A}$ stands for the collection of all $A$-element subsets of $P_{0}^{k}$. In particular, these random variables have common expectation, which equals
\begin{displaymath} \sum_{P \in \calP^{A}} \tn_{k}\{P \subset U\}. \end{displaymath}
The $\tn_{k}$-probability that any fixed $A$-element subset $P \subset P_{0}^{k}$ is contained in $U$ is
\begin{displaymath} \tn_{k}\{P \subset U\} = \left(\frac{|U|}{|R|}\right)^{A} = (2e)^{A}(\delta m)^{-As}. \end{displaymath}
Since $\card \calP^{A} = \binom{k}{A}$, we conclude that
\begin{displaymath} \tilde{\E}\left[\binom{Z_{S}}{A}\right] = \binom{k}{A}(2e)^{A}(\delta m)^{-As} \leq \left(\frac{e \cdot 2(\delta m)^{s}}{A}\right)^{A}(2e)^{A}(\delta m)^{-As} = \left(\frac{4e^{2}}{A}\right)^{A} \leq \frac{1}{100}. \end{displaymath}
for a large enough absolute choice of $A \in \N$.

Next, we consider the variance of $\binom{Z_{S}}{A}$. Estimating crudely,
\begin{displaymath} \widetilde{\Var}\left[\binom{Z_{S}}{A}\right] \leq \tilde{\E}\left[\binom{Z_{S}}{A}^{2}\right] \lesssim_{A} \tilde{\E}[Z_{S}^{2A}] = \int_{0}^{\infty} t^{2A - 1}\tilde{\tn}\{Z_{S} \geq t\} \, dt. \end{displaymath}
Recalling that $Z_{S} \sim \text{Bin}(k,2e|S|/|R|)$, where $k \leq 2|R|/|S|$, it is easy to check (using directly the formula for the probability density function of $Z_{S}$) that the integral above admits a bound $C_{A} < \infty$ depending only on $A$. 

Since the random variables $\binom{Z_{S}}{A}$ are independent for various $S \in \calS_{R}$, the variance of their sum is the sum of their variance. Hence, Markov's inequality is a useful tool for bounding the probabilities related to the average
\begin{displaymath} \tilde{X}_{R} := \frac{1}{(\delta m)^{s}} \sum_{S \in \calS_{R}} \binom{Z_{S}}{A}. \end{displaymath}
Recalling that $\tilde{\E}[\tilde{X}_{R}] = \tilde{\E}[Z_{S}] \leq 1/100$ for large enough $A$, we obtain
\begin{align*} \tilde{\tn}\{\tilde{X}_{R} \geq 1/10\} & \leq \tilde{\tn}\{\tilde{X}_{R} - \E_{k}[\tilde{X}_{R}] \geq 1/100]\}\\
& \leq \tilde{\tn}\{(\tilde{X}_{R} - \E_{k}[\tilde{X}_{R}])^{2} \geq 10^{-4}\}\\
& \leq 10^{4} \cdot \widetilde{\Var}[\tilde{X}_{R}] = \frac{10^{4}}{(\delta m)^{s}} \widetilde{\Var}[Z_{S}] \leq \frac{10^{4}C_{A}}{(\delta m)^{s}}. \end{align*} 
Taking $m$ large enough (depending on $\delta$) and using \eqref{form11}, this gives
\begin{displaymath} \tn_{k}\{X_{R} \geq 1/10\} \leq \tilde{\tn}\{\tilde{X}_{R} \geq 1/10\} < \delta^{s}/20. \end{displaymath}
This holds uniformly for $0 \leq k \leq 2(\delta m)^{s}$, so the sum in \eqref{fixk} is also $< \delta^{s}/20$. Finally (as discussed above \eqref{fixk}), taking $m$ so large that also $\tn\{|P_{0} \cap R| > 2(\delta m)^{s}\} < \delta^{s}/20$, we obtain $\tn\{X_{R} \geq 1/10\} < \delta^{s}/20$, and finally
\begin{displaymath} \tn\left\{\max_{R \in \calU} X_{R} \geq \frac{1}{10} \right\} < \frac{1}{10}. \end{displaymath}
This concludes the proof of the lemma. \end{proof}


Now we are prepared for the proof of the main lemma.

\begin{proof}[Proof of Lemma \ref{mainLemma}] Informally speaking, the main question to answer is the following: "In expectation, how many large subsets of $P_{0}$ land in tubes $T$ of width $1/m \leq w(T) \leq \delta$?" Now, we set to formalise and answer this question.

Fix a number $\beta = 2^{j}$, $j \geq 0$, such that $1/m \leq \beta/m \leq \delta/r^{2}$, where $r \geq 1$ is a constant depending only on the dimension $d$, the meaning of which will be clarified soon. If $T \subset \R^{d}$ is a tube and $h > 0$, denote by $hT$ the tube with the same central line as $T$ but with $w(hT) = hw(T)$. We know that any tube $r^{2}T$ with $w(T) = \beta/m$ has width $\leq \delta$, hence meets at most $k$ cubes in $\calU$ by assumption, so if $p$ is a random point in $U$, we have
\begin{displaymath} \tn\{p \in 2r^{2}T\} = \frac{|2r^{2}T \cap U|}{|U|} \leq C\frac{k \cdot \delta \cdot (\beta/m)^{d - 1}}{\delta^{d-s}} = Ck\delta^{s - d + 1} \cdot \left(\frac{\beta}{m}\right)^{d - 1} \end{displaymath}
for some suitable absolute constant $C \geq 1$.

The purpose of the constant $r \geq 1$ is the following. Assume that we have already determined the value of $\epsilon \sim 1/m$ (the side-length of the cubes in $\calQ$), and imagine placing cubes of side-length $\epsilon$ centred at each point in $P_{0}$: denote these cubes by $\calQ_{0}$. We require $r$ to be so large that the following conditions are satisfied: $r \geq 2\epsilon m$, and if $w(T) \geq 1/m$, and $T$ intersects one of the cubes in $\calQ_{0}$, then, $rT$ contains the centre, a point in $P_{0}$. These conditions are satisfied by a large constant $r$, the size of which depends only on $d$ and the \textbf{absolute} constants in $\epsilon \sim 1/m$. As a consequence of the second condition, if $T$ is a tube of width $w(T) = \beta/m$, which meets $> k \cdot (m w(T))^{s} = k \cdot \beta^{s}$ cubes in $\calQ_{0}$ for some (large) $k$, it follows that $rT$ must contain a $k \cdot \beta^{s}$ element subset of $P_{0}$, and we wish to estimate the probability of this event. For technical reasons, however, we choose to estimate the probability of $\geq k \cdot \beta^{s}$ points being contained in $r^{2}T$ instead of $rT$. 

Write $q = k \cdot \beta^{s}$ and assume that $q \in \N$ (if not, everything below would verbatim with $q := \ceil{k \cdot \beta^{s}}$ instead). Given any random $q$-element subset $\{p_{1},\ldots,p_{q}\}$ of $U$, we have
\begin{equation}\label{form13} \tn\{\{p_{1},\ldots,p_{q}\} \subset 2r^{2}T\} \leq (Ck\delta^{s - d + 1})^{q} \cdot \left(\frac{\beta}{m}\right)^{q(d - 1)}. \end{equation}
Next, we record the following simple geometric fact: for any $0 < \tau \leq 1$, one can pick $\sim \tau^{-2(d - 1)}$ "representative" tubes of width $2\tau$ such that the intersection of an arbitrary tube of width $\tau$ with $[0,1]^{d}$ is contained in one of these representatives. Applying \eqref{form13} to each of these "representative" tubes (with $\tau = \beta/m$) and using the union bound yields
\begin{displaymath} \tn\{\{p_{1},\ldots,p_{q}\} \subset r^{2}T \text{ for some } T \text{ with } w(T) = \beta/m\} \leq (Ck\delta^{s - d + 1})^{q} \cdot \left(\frac{\beta}{m}\right)^{(q - 2)(d - 1)}. \end{displaymath}
Let us emphasise that $T$ above stands for an arbitrary tube, and the notion of "representatives" was only used as a tool to reach the bound. 

Recall that $P_{0}$ was a random $m^{s}$-element subset of $U$. Write $P_{0} := \{p_{1},\ldots,p_{m^{s}}\}$, and denote by $\calP_{q}$ all the $q$-element subsets of $\{1,\ldots,m^{s}\}$. Given $B \in \calP_{q}$, let $B(P_{0}) := \{p_{j} : j \in B\}$. This is a random $q$-element subset of $U$, and the probabilistic bound above can be applied as follows: 
\begin{align} \E & \left[\sum_{B \in \calP_{q}} \1_{\{B(P_{0}) \subset r^{2}T  \text{ for some } T \text{ with } w(T) = \beta/m\}} \right] \notag \\
& = \sum_{B \in \calP_{q}} \tn\{B(P_{0}) \subset r^{2}T  \text{ for some } T \text{ with } w(T) = \beta/m\} \notag \\
& \leq \binom{m^{s}}{q}(Ck\delta^{s - d + 1})^{q} \cdot \left(\frac{\beta}{m}\right)^{(q - 2)(d - 1)} \notag \\
& \leq (eCk\delta^{s - d + 1})^{q} \cdot \left(\frac{m^{s}}{q}\right)^{q} \cdot \left(\frac{\beta}{m}\right)^{(q - 2)(d - 1)} \notag \\
& = (eC\delta^{s - d + 1})^{q} \cdot \left(\frac{m^{s}}{\beta^{s}}\right)^{q} \cdot \left(\frac{\beta}{m}\right)^{(q - 2)(d - 1)} \notag \\
& = (eC\delta^{s - d + 1})^{q} \cdot \left(\frac{\beta}{m}\right)^{q(d - 1 - s) - 2(d - 1)} \notag \\
&\label{form2} =: D^{k \cdot \beta^{s}} \cdot \left(\frac{\beta}{m}\right)^{k \cdot \beta^{s}(d - 1 - s) - 2(d - 1)}. \end{align}
Next, for reasons to become apparent shortly, we wish to estimate the sum of the numbers in \eqref{form2} over $\beta = 2^{j}$ such that $\beta/m \leq \delta^{(d - s)/d}m^{-s/d}$, that is, for
\begin{displaymath} \beta \leq \delta^{(d - s)/d}m^{1 - s/d} \leq m^{1 - s/d}. \end{displaymath}
Observe that then $\beta/m \leq \delta/r^{2}$ for large enough $m$, which was necessary for the estimates above. We are free to choose $k = k_{s} \in \N$ at will, and the first requirement we place is that
\begin{displaymath} k(d - 1 - s) - 2(d - 1) > 0. \end{displaymath}
Then also $k \cdot \beta^{s}(d - 1 - s) - 2(d - 1) > 0$ for all $\beta = 2^{j} \geq 1$, so we may plug in the upper bound for $\beta$ to obtain
\begin{align*} \sum_{\beta = 2^{j} = 1}^{m^{1 - s/d}} D^{k \cdot \beta^{s}} \left(\frac{\beta}{m}\right)^{k \cdot \beta^{s}(d - 1 - s) - 2(d - 1)} \leq \sum_{\beta = 2^{j} = 1}^{m^{1 - s/d}} D^{k \cdot \beta^{s}} \cdot m^{-sk\cdot \beta^{s}(d - 1 - s)/d + 2s(d - 1)/d}. \end{align*}
Next, we note that the various numbers $\beta^{s} = 2^{js}$ are separated by $\geq 1$ for $j$ large enough (depending on $s$), so we may replace the original summation over $\beta = 2^{j} \leq m^{1 - s/d}$ by a summation over $\beta \in \N$, with the gain of replacing $\beta^{s}$ by $\beta$ in the process. This may cost us a multiplicative constant $C_{s} \geq 1$ depending on $s$. The result is the following geometric sum:
\begin{align} \ldots \leq C_{s}\sum_{\beta = 1}^{\infty} & D^{k \cdot \beta} \cdot m^{-sk \cdot \beta(d - 1 - s)/d + 2s(d - 1)/d} \notag\\
& \leq m^{2s(d - 1)/d} \cdot C_{s} \sum_{\beta = 1}^{\infty} \left(D^{k} \cdot m^{-sk(d - 1 - s)/d}\right)^{\beta} \notag \\
&\label{form3} =  \frac{C_{s} \cdot D^{k} \cdot m^{-sk(d - 1 - s)/d}}{1 - D^{k} \cdot m^{-sk(d - 1 - s)/d}} \cdot m^{2s(d - 1)/d} \notag\\
& = \left[\frac{C_{s} \cdot D^{k} \cdot m^{-sk(d - 1 - s)/d + s(1 - 2/d)}}{1 - D^{k} \cdot m^{-sk(d - 1 - s)/d}} \right] \cdot m^{s},  \end{align}
where the upshot is that the factor in front of of $m^{s} = |P_{0}|$ can be made arbitrarily small by choosing $m$ large enough (depending on $\delta$ via the definition of $D$), and taking $k = k_{s}$ so large that $s(1 - 2/d) - sk(d - 1 - s)/d < 0$.

Let us sum up what we have gained so far. Choosing $m \in \N$ large enough depending on $\delta$, $k$ large enough depending on $s$ and $A \geq 1$ large enough but absolute, we can now conclude (see explanations below) that the following three events each hold with probability at least $9/10$:
\begin{equation}\label{form5} \min_{R \in \calU} |P_{0} \cap R| \geq \frac{(\delta m)^{s}}{2}, \end{equation}
\begin{equation}\label{form4a} \max_{R \in \calU} \sum_{S \in \calS_{R}} \sum_{P \in \calP_{A}} \1_{\{P \subset S\}} \leq \frac{(\delta m)^{s}}{8}, \end{equation}
and
\begin{equation}\label{form4} \sum_{\beta = 2^{j} = 1}^{(\delta m)^{1 - s/d}} \sum_{B \in \calP_{k \cdot \beta^{s}}} \1_{\{B(P_{0}) \subset r^{2}T \text{ for some } T \text{ with } w(T) = \beta/m\}} \leq \frac{(\delta m)^{s}}{8}, \end{equation}
where $\calS_{R}$ in \eqref{form4a} stands for the grid of cubes of side-length $\delta^{(d - s)/d}m^{-s/d}$ defined above Lemma \ref{mainClaim}. First, \eqref{form5} has high probability (when $m = m_{\delta}$ is large enough) simply because the random variables $|P_{0} \cap R|$, $R \in \calU$, are distributed $\sim \text{Bin}(m^{s},\delta^{s})$, and $\card \calU = \delta^{-s}$ does not grow as $m \to \infty$. Second, the fact that \eqref{form4a} holds with high probability is just another way of writing the conclusion of Claim \ref{mainClaim}, since
\begin{displaymath} \sum_{P \in \calP_{A}} \1_{\{P \subset S\}} = \binom{|P_{0} \cap S|}{A}. \end{displaymath}
Third, the situation that \eqref{form4} has probability $\geq 9/10$ can be reached by taking $m = m_{\delta}$ and $k = k_{s}$ so large that the expectation of the sum in \eqref{form4} is bounded by $(\delta m)^{s}/100$: this is possible by the bound \eqref{form3}.

Since the three events corresponding to \eqref{form5}--\eqref{form4} each hold with probability $\geq 9/10$, all of them hold simultaneously with positive probability. So, we can and will choose a set $P_{0}$ satisfying all three conditions. Next, to obtain a "regularised" subset $\tilde{P} \subset P_{0}$, we execute the following point removal process (PRP):
\begin{itemize}
\item[(i)] From each $A$-element subset of $P_{0}$ contained in a cube $S \in \calS_{R}$, for any $R \in \calU$, remove one point. 
\item[(ii)] For all $\beta = 2^{j} \in [1,(\delta m)^{1 - s/d}]$ and from all $(k \cdot \beta^{s})$-element subsets contained in a tube $r^{2}T$ with $w(T) = \beta/m$, remove one point.
\end{itemize}
The remaining set is denoted by $\tilde{P}$. The first observation is that 
\begin{displaymath} |\tilde{P} \cap R| \geq \frac{(\delta m)^{s}}{4} \end{displaymath}
for every cube $R \in \calU$. Indeed, by \eqref{form4a}, the number of $A$-element subsets contained in some cube $S \in \calS_{R}$ does not exceed $(\delta m)^{s}/8$ for any $R \in \calU$, so PRP(i) deletes at most $(\delta m)^{s}/8$ points inside each $R \in \calU$. By \eqref{form4}, PRP(ii) deletes at most $(\delta m)^{s}/8$ points from $P_{0}$ altogether. Since $|P_{0} \cap R| \geq (\delta m)^{s}/2$ for all $R \in \calU$ by \eqref{form5}, the claim follows.

\subsection{Conditions (a)--(c)} Next, we will remove some further points from $\tilde{P}$ to ensure that (a)--(c) are satisfied. We already know by PRP(i) that each of the cubes $S \in \calS_{R}$ of side-length $\delta^{(d - s)/d}m^{-s/d}$ contains fewer than $A$ points in $\tilde{P}$. It follows that we may remove points from $\tilde{P}$ until the following two conditions are met: the number of points remaining in each cube $R \in \calU$ is $\gtrsim (\delta m)^{s}/A$, and the pairwise distance between the remaining points is at least
\begin{equation}\label{distance} 5d \cdot \delta^{(d - s)/d}m^{-s/d}. \end{equation}
As a final "regularisation", we remove yet more points in order make sure that each cube $R \in \calU$ contains the same number, say $N$, of points, and this number satisfies $(\delta m)^{s}/A \lesssim N \leq (\delta m)^{s}$. The subset of $\tilde{P}$ so obtained is called $P$.

The side-length $\epsilon \sim 1/m$ of the cubes in $\calQ$, centred at the points in $P$, is now determined by the requirement (b), saying that each cube $R \in \calU$ should contain $(\delta/\epsilon)^{s}$ cubes in $\calQ$. Since each such $R$ contains $N \gtrsim (\delta m)^{s}/A$ points of $P$, and $A$ is an absolute constant, we may choose $\epsilon \sim 1/m$ so that $(\delta/\epsilon)^{s} = N$. In particular, $1/\epsilon \leq m$, since $(\delta/\epsilon)^{s} = N \leq (\delta m)^{s}$.

It remains to verify that $\calQ$ satisfies the requirements (c) and (d). Condition (c), that an arbitrary cube of side-length $\delta^{(d - s)/d}m^{-s/d}$ intersects at most one the cubes in $\calQ$, follows immediately from \eqref{distance}, and the fact that $\epsilon \sim 1/m$ is far smaller than $5d \cdot \delta^{(d - s)/d}m^{-s/d}$ for large $m \in \N$. This also implies that the cubes in $\calQ$ are all disjoint.

\subsection{Condition (d)} To prove (d), we split into three cases. First, every tube $T$ of width $w(T) = 2\epsilon$ intersects fewer than $k$ cubes $Q$. Otherwise $rT$ would contain $k$ points in $P$ by the choice of $r$, and we could pick a tube $T'$ with $w(T') = 1/m$, $T = (2\epsilon m)T'$, so that a $k$-element subset of $P$ is contained in $r^{2}T' \supset rT$ (using $r \geq 2\epsilon m)$. This would contradict the PRP in the case $\beta = 1$. 

Second, fix any tube $T$ with $\epsilon \leq w(T) \leq \delta^{(d - s)/d}m^{-s/d}$ and locate $\beta = 2^{j} \in [1,(\delta m)^{(d - s)/d}]$ with the property that $ \beta \leq w(T)/\epsilon \leq 2\beta$ (here we need that $1/\epsilon \leq m$). By PRP(ii), every tube $rT'$ with $w(T') = \beta/m$ can only contain $\leq k \cdot \beta^{s}$ points in $P$. Since $w(T) \lesssim \beta/m$, it follows that $rT$ can also contain at most $\lesssim k \cdot \beta^{s}$ points in $P$. Then, by the choice of $r$, the tube $T$ can only meet $\lesssim k \cdot \beta^{s}$ cubes in $\calQ$.

Finally, fix a tube $T$ with $\delta^{(d - s)/d}m^{-s/d} \leq w(T) \leq \delta$. Then, since the cubes $S \in \calS_{R}$, $R \in \calU$, have side-length $\delta^{(d - s)/d}m^{-s/d} \leq w(T)$, we see that each $S$ intersecting $2T$ is contained in $r_{d}T$ for a suitable dimensional constant $r_{d} \geq 1$. Fixing any one cube $R \in \calU$, we obtain
\begin{align*} \delta^{d - s}m^{-s} & \cdot \card\{S \in \calS_{R} : S \cap 2T \neq \emptyset\}\\
& \leq |S| \cdot \card\{S \in \calS_{R} : S \subset r_{d}T\}\\
& \leq |r_{d}T \cap R| \lesssim \delta \cdot w(T)^{d - 1}. \end{align*} 
Since $2T$ can only intersect at most $k$ cubes $R \in \calU$, it follows that
\begin{align*} \card\left\{S \in \bigcup_{R \in \calU} \calS_{R} : S \cap 2T \neq \emptyset\right\} & \lesssim k \cdot \delta^{s - d + 1}m^{s}w(T)^{d - 1}\\
& = k \cdot \left(\frac{w(T)}{\delta}\right)^{d - 1 - s} \cdot (mw(T))^{s}\\
& \lesssim k \cdot \left(\frac{w(T)}{\epsilon}\right)^{s}.  \end{align*}
Now, if $T$ intersects a cube $Q \in \calQ$, then $2T$ meets the cube $S$ containing the centre of $Q$. On the other hand, each cube $S$ contains only one such centre by (c), and so
\begin{displaymath} \card\{Q \in \calQ : Q \cap T \neq \emptyset\} \lesssim k \left(\frac{w(T)}{\epsilon}\right)^{s}. \end{displaymath}
Thus, the set $P$ satisfies (c) and (d), and the proof of the lemma is complete. \end{proof}

\section{Proof of Theorem \ref{main}} 

With Lemma \ref{mainLemma} at our disposal, the proof of Theorem \ref{main} is straightforward:

\begin{proof}[Proof of Theorem \ref{main}] Fix $s < d - 1$ and let $k = k_{s}$ be the corresponding constant from Lemma \ref{mainLemma}. We will need to construct a compact set $K \subset \R^{d}$ with $\calH^{s}(K) > 0$ and satisfying the tube condition \eqref{tubes} with $t = s$. This is achieved by first defining recursively a sequence of families of closed disjoint cubes $\calQ_{n}$. We will maintain the invariant that the families $\calQ_{n}$ should always satisfy the hypotheses of Lemma \ref{mainLemma}. Clearly the initial collection of cubes $\calQ_{0} := \{[0,1]^{d}\}$ has this property with $\delta = 1$. Applying the lemma with $\calU = \calQ_{0}$, we obtain a family of cubes $\calQ_{1} := \calQ$, which again satisfies the assumptions of the lemma. Proceeding this way, we may define a compact set $K$ by
\begin{displaymath} K = \bigcap_{n = 0}^{\infty} \bigcup_{Q \in \calQ_{n}} Q. \end{displaymath}
Given any cube $Q \in \calQ_{n}$, the following properties of the cubes in $\calQ_{n}$ and $\calQ_{n + 1}(Q) := \{Q' \in \calQ_{n + 1} ; Q' \subset Q\}$, $Q \in \calQ_{n}$, are easy consequences of Lemma \ref{mainLemma}:
\begin{displaymath} \sum_{Q' \in \calQ_{n + 1}(Q)} d(Q')^{s} = d(Q)^{s}, \end{displaymath}
and for any ball $B$ with $d(B) \geq \ell_{n}$ -- the common side-length of the cubes in $\calQ_{n}$ -- one has
\begin{displaymath} \sum_{Q \in \calQ_{n} : Q \cap B \neq \emptyset} d(Q)^{s} \leq Cd(B)^{s} \end{displaymath}
for some absolute constant $C \geq 1$. The first property is precisely Lemma \ref{mainLemma}(b). The second property is a condition far weaker than Lemma \ref{mainLemma}(d), which even implies that the same bound remains valid, if on the left hand side the ball $B$ is replaced by a tube $T \supset B$ of width $w(T) = d(B)$. The two properties have the consequence (see for instance \cite[\S4.12]{Ma}) that $0 < \calH^{s}(K) < \infty$, and in fact each cube $Q \in \calQ_{n}$ has
\begin{equation}\label{form7} \calH^{s}(K \cap Q) \sim d(Q)^{s} \sim \ell_{n}^{s} \end{equation}
with implicit constants independent of $n$; in fact, we only need $\lesssim$ in \eqref{form7}, which follows by using the natural covers for $K \cap Q$.

Now we are prepared to prove the estimate $\calH^{s}(K \cap T) \lesssim w(T)^{s}$ for any given tube $T \subset \R^{d}$. If $w(T) \geq 1$, there is nothing to show, so assume that $w(T) < 1$ and fix $n \in \N$ such that $\ell_{n} < w(T) \leq \ell_{n - 1}$. Then, Lemma \ref{mainLemma}(d) tells us that $T$ meets no more than
\begin{displaymath} \lesssim \left(\frac{w(T)}{\ell_{n}}\right)^{s} \end{displaymath}
cubes in $\calQ_{n}$. In particular, by \eqref{form7}, we have
\begin{displaymath} \calH^{s}(K \cap T) \leq \sum_{Q \in \calQ_{n} : Q \cap T \neq \emptyset} \calH^{s}(K \cap Q) \lesssim \left(\frac{w(T)}{\ell_{n}}\right)^{s} \cdot \ell_{n}^{s} = w(T)^{s}. \end{displaymath}
This completes the proof. \end{proof}

\section{Open problems}

\begin{itemize}
\item (Suggested by V. Suomala) The set $K$ constructed for Theorem \ref{main} is far from being Ahlfors $s$-regular. Does it have to be so?
\item Is it possible to construct a set $K \subset \R^{d}$ with $0 < \calH^{d - 1}(K) < \infty$ such that
\begin{displaymath} \calH^{d - 1}(K \cap T) \leq C_{s}w(T)^{s} \end{displaymath}
for all tubes $T \subset \R^{d}$ and for all $s < d - 1$ simultaneously? More specifically, what happens with $(d - 1)$-dimensional self-similar sets, which contain "irrationality" in the rotational components of the generating similitudes? 
\end{itemize}

\end{document}